\documentclass{article}
\usepackage{amsfonts}
\usepackage{hyperref}

\newtheorem{theorem}{Theorem}

\newenvironment{proof}[1][Proof]{\noindent\textbf{#1.} }{\ \rule{0.5em}{0.5em}}
\input{tcilatex}

\begin{document}

\title{A note on the Jacobi stability of dynamical systems via Lagrange
geometry and KCC theory}
\author{Mircea Neagu\thanks{%
Transilvania University of Bra\c{s}ov, Romania.} \ and Elena Ovsiyuk\thanks{%
Mozyr State Pedagogical University named after I.P. Shamyakin, Belarus.}}
\date{}
\maketitle

\begin{abstract}
In this paper, via the least squares variational method, we develop the
Lagrange geometry (in the sense of nonlinear connection, $d-$torsions and
the deviation curvature tensor) and the KCC theory for a given dynamical
system. Further, a result on the Jacobi stability of the solutions of this
dynamical system is also obtained.
\end{abstract}



\section{Lagrange geometry and KCC theory for a given dynamical system}

Let $M$ be a $n$-dimensional smooth manifold, whose coordinates are $\left(
x^{i}\right) _{i=\overline{1,n}}.$ Let $TM$ be the tangent bundle, whose
coordinates are $\left( x^{i},y^{i}\right) _{i=\overline{1,n}}$, and let us
consider a vector field $X=\left( X^{i}(x)\right) _{i=\overline{1,n}}$ on $M$%
, which produces the dynamical system%
\begin{equation}
\frac{dx^{i}}{dt}=X^{i}(x(t)),\quad i=\overline{1,n}.  \label{DS}
\end{equation}

It is obvious that the solutions of class $C^{2}$ of the dynamical system (%
\ref{DS}) are the global minimum points for the \textit{least squares
Lagrangian} $L:TM\rightarrow \mathbb{R}$, given by (see Balan-Neagu \cite%
{Bal-Nea})%
\begin{equation}
L(x,y)=\delta _{ij}\left( y^{i}-X^{i}(x)\right) \left( y^{j}-X^{j}(x)\right)
.  \label{LSL}
\end{equation}%
The Euler-Lagrange equations of (\ref{LSL}) are expressed by ($i=\overline{%
1,n}$)%
\begin{equation}
\frac{\partial L}{\partial x^{i}}-\frac{d}{dt}\left( \frac{\partial L}{%
\partial y^{i}}\right) =0\Leftrightarrow \frac{d^{2}x^{i}}{dt^{2}}%
+2G^{i}(x,y)=0,  \label{E-L}
\end{equation}%
where $y^{i}=dx^{i}/dt$ and%
\[
G^{i}(x,y)=\frac{1}{4}\left( \frac{\partial ^{2}L}{\partial x^{j}\partial
y^{i}}y^{j}-\frac{\partial L}{\partial x^{i}}\right) =-\frac{1}{2}\left[
\left( \frac{\partial X^{i}}{\partial x^{j}}-\frac{\partial X^{j}}{\partial
x^{i}}\right) y^{j}+\frac{\partial X^{j}}{\partial x^{i}}X^{j}\right] 
\]%
is endowed with the geometrical meaning of \textit{semispray} of $L$. The
preceding semispray allows us to construct a whole collection of Lagrangian
geometrical objects (such as nonlinear connection and $d-$torsions) that
characterize the initial dynamical system (\ref{DS}) and that can be studied
further via the Kosambi-Cartan-Chern (KCC) theory for SODEs. For more
details on these topics, see Miron-Anastasiei \cite{Mir-An}, Udri\c{s}%
te-Neagu works \cite{Udr} and \cite{Nea-Udr}, B\"{o}hmer et al. \cite{Bohmer}
and Buc\u{a}taru-Miron \cite[pp. 71-72]{Buc-Mir}.

It is important to note that the above Lagrange geometry produced by the
Lagrangian (\ref{LSL}) is exposed in details in the monograph \cite[pp.
129-133]{Bal-Nea}. If we use the notation $J=\left( \partial X^{i}/\partial
x^{j}\right) _{i,j=\overline{1,n}}$ for the Jacobian matrix of $X$, then
this Lagrange geometry is achieved via the nonzero geometrical objects:

\begin{itemize}
\item $N=\left( N_{j}^{i}\right) _{i,j=\overline{1,n}}=-\dfrac{1}{2}\left[
J-J^{t}\right] $ is the \textit{Lagrangian nonlinear connection}, where%
\[
N_{j}^{i}=\frac{\partial G^{i}}{\partial y^{j}}=-\frac{1}{2}\left( \frac{%
\partial X^{i}}{\partial x^{j}}-\frac{\partial X^{j}}{\partial x^{i}}\right)
; 
\]

\item $R_{k}=\left( R_{jk}^{i}\right) _{i,j=\overline{1,n}}=\dfrac{\partial N%
}{\partial x^{k}},$ $\forall $ $k=\overline{1,n},$ are the \textit{%
Lagrangian }$d-$\textit{torsions}, where%
\[
R_{jk}^{i}={\frac{\delta N_{j}^{i}}{\delta x^{k}}}-\frac{\delta N_{k}^{i}}{%
\delta x^{j}},\quad {\frac{\delta }{\delta x^{k}}={\frac{\partial }{\partial
x^{k}}}-N_{k}^{r}{\frac{\partial }{\partial y^{r}};}} 
\]

\item $P=\left( P_{j}^{i}\right) _{i,j=\overline{1,n}}=R_{k}y^{k}+\mathcal{E}
$ is the \textit{deviation curvature tensor }which is given by the formula%
\[
P_{j}^{i}=-2\frac{\partial G^{i}}{\partial x^{j}}-2G^{l}\frac{\partial
N_{j}^{i}}{\partial y^{l}}+\frac{\partial N_{j}^{i}}{\partial x^{l}}%
y^{l}+N_{l}^{i}N_{j}^{l}=R_{jk}^{i}y^{k}+\frac{\delta \mathcal{E}^{i}}{%
\delta x^{j}}, 
\]%
where%
\[
\mathcal{E}^{i}=2G^{i}-N_{j}^{i}y^{j}=-\frac{1}{2}\left( \frac{\partial X^{i}%
}{\partial x^{j}}-\frac{\partial X^{j}}{\partial x^{i}}\right) y^{j}-\frac{%
\partial X^{j}}{\partial x^{i}}X^{j} 
\]%
is the\textit{\ first invariant of the semispray} of the Lagrangian (\ref%
{LSL}). Here we have%
\[
\mathcal{E=}\left( \frac{\delta \mathcal{E}^{i}}{\delta x^{j}}\right) _{i,j=%
\overline{1,n}}, 
\]%
where%
\[
\frac{\delta \mathcal{E}^{i}}{\delta x^{j}}=-\frac{1}{2}\left( \frac{%
\partial ^{2}X^{i}}{\partial x^{j}\partial x^{k}}-\frac{\partial ^{2}X^{k}}{%
\partial x^{i}\partial x^{j}}\right) y^{k}-\frac{\partial ^{2}X^{k}}{%
\partial x^{i}\partial x^{j}}X^{k}-\frac{\partial X^{k}}{\partial x^{i}}%
\frac{\partial X^{k}}{\partial x^{j}}- 
\]%
\[
-\frac{1}{4}\left( \frac{\partial X^{k}}{\partial x^{j}}-\frac{\partial X^{j}%
}{\partial x^{k}}\right) \left( \frac{\partial X^{i}}{\partial x^{k}}-\frac{%
\partial X^{k}}{\partial x^{i}}\right) . 
\]
\end{itemize}

In the background of KCC theory from the paper \cite[pp. 10-12]{Bohmer} we
note that the solutions of the Euler-Lagrange equations (\ref{E-L}) are
Jacobi stable iff the real parts of the eigenvalues of the deviation tensor $%
P$ are strictly negative everywhere, and Jacobi unstable, otherwise. The
Jacobi stability or instability has the geometrical meaning that the
trajectories of the Euler-Lagrange equations (\ref{E-L}) are bunching
together or are dispersing. As a consequence of all above we infer

\begin{theorem}
If the dimension $n\geq 2$ is odd and the matrix $\mathcal{E}$ is
skew-symmetric, then the dynamical system (\ref{DS}) is Jacobi unstable.
\end{theorem}

\begin{proof}
If the matrix $\mathcal{E}$ is skew-symmetric it follows that the deviation
tensor matrix $P$ is also skew-symmetric because the Lagrangian $d-$torsion
matrices are skew-symmetric (cf. above formulas). Consequently, the
condition of odd dimensionality implies that the matrix $P$ has its
determinant equal to zero. In other words, the value $\lambda =0$ is an
eigenvalue for the deviation tensor matrix $P.$

It is obvious now that we obtain what we were looking for.
\end{proof}

\end{document}